\documentclass[11pt]{article}
\usepackage{eurosym}
\usepackage{amssymb}
\usepackage{amsfonts}
\usepackage{amsmath}
\usepackage{graphicx}
\usepackage{hyperref}
\usepackage{subcaption}
\usepackage{cancel}
\usepackage{tikz}
\usepackage{float}
\usepackage{cite}
\usepackage{float}

\newtheorem{theorem}{Theorem}

\newtheorem{lemma}[theorem]{Lemma}

\newtheorem{proposition}[theorem]{Proposition}
\newtheorem{remark}[theorem]{Remark}

\newenvironment{proof}[1][Proof]{\textbf{#1.} }{\ \rule{0.5em}{0.5em}}
\begin{document}

\author{George Avalos \\
Department of Mathematics, University of Nebraska-Lincoln, USA \and Pelin G. Geredeli \\
Department of Mathematics, Iowa State University, USA }
\title{ A Resolvent Criterion Approach to Strong Decay of a Multilayered Lamé-Heat System
}
\maketitle

\begin{abstract}

We consider a multilayer hyperbolic-parabolic PDE system which constitutes a coupling of 3D thermal - 2D elastic - 3D elastic dynamics, in which the boundary interface coupling between 3D fluid and 3D structure is realized via a 2D elastic equation. Our main result here is one of strong decay for the given multilayered - heat system. That is, the solution to this composite PDE system is stabilized asymptotically to the zero state. 

Our proof of strong stability takes place in the ``frequency domain" and ultimately appeals to the pointwise resolvent condition introduced by Tomilov \cite{Tomilov}. This very useful result, however, requires that the semigroup associated with our multilayered FSI system be completely non-unitary (c.n.u).  Accordingly, we firstly establish that the semigroup $\{e^{\mathcal {A}t}\}_{t\geq 0}$ is indeed c.n.u., in part by invoking relatively recent results of global uniqueness for overdetermined Lamé systems on nonsmooth domains. Although the entire proof also requires higher regularity results for some trace terms, this \textit{``resolvent criterion approach"} allows us to establish a ``classially soft" proof of strong decay. In particular, it avoids the sort of technical PDE multipliers invoked in \cite{jde}.

\vskip.3cm \noindent \textbf{Key terms:} Fluid-Structure Interaction, Lamé-Heat System, Semigroup, Strong Stability 

\end{abstract}

\bigskip 

\section{Introduction}

\subsection{Description of the Problem}

The multi-layered PDE models discussed below arise in the context of fluid–structure interaction with composite structures. Such FSI mathematically account for the fact that mammalian veins and arteries are typically composed of various layers of tissues; each layer will manifest its own intrinsic material properties, and are moreover separated from the others by thin elastic laminae; see \cite{multi-layered}. Consequently, appropriate FSI will contain an additional PDE which evolves on the boundary interface to account for thin elastic layer. \\

\noindent In what follows we describe the setting and explicit description of the PDE system under the study: \\

\noindent Throughout, the fluid geometry $\Omega _{f}$ $\subseteq \mathbb{R}^{3}$ will be a
Lipschitz, bounded domain with exterior boundary $\Gamma _{f}$. The
polyhedral structure domain $\Omega _{s}$ $\subseteq \mathbb{R}^{3}$ will be
\textquotedblleft completely immersed\textquotedblright\ in $\Omega _{f},$
with its polygonal boundary faces denoted $\Gamma _{j},$ $1\leq j\leq K$. If given faces 
$\Gamma _{i}$ and $\Gamma _{j}$ satisfy $\Gamma _{i}\cap \Gamma _{j}\neq
\emptyset $ for $i\neq j$ then the interior dihedral angle between them is
in $(0,2\pi )$ (see Figure.)

\begin{center}
\includegraphics[scale=0.4]{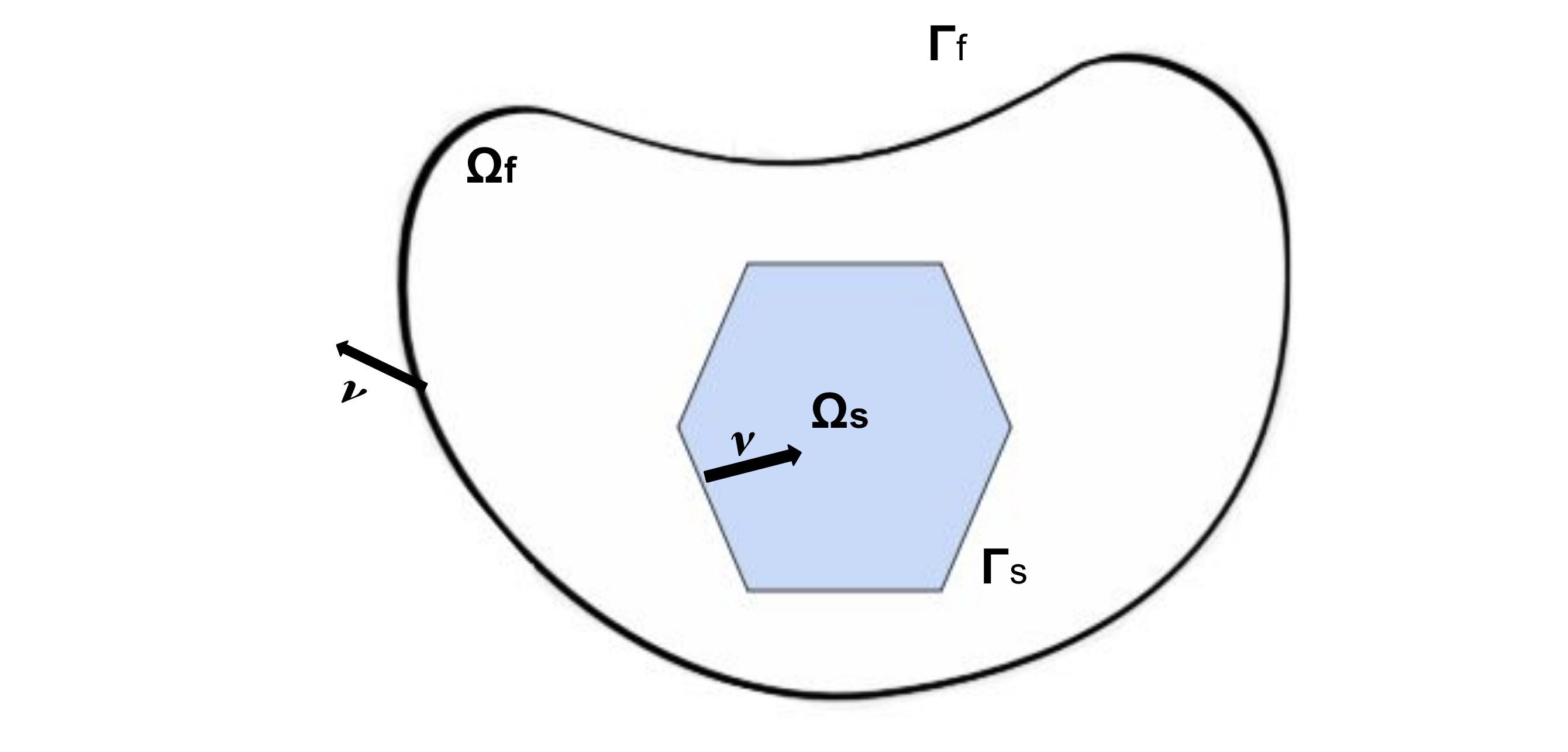} 

\textbf{Figure: Geometry of the FSI Domain} 
\end{center}
The boundary interface $\Gamma _{s}=\partial \Omega _{s}$ between $\Omega
_{f}$ and $\Omega _{s}$ is then the union of said polygonal faces. That is, 
\[
\Gamma _{s}=\cup _{j=1}^{K}\overline{\Gamma }_{j}.
\]
In addition, $\nu (x)$ is the unit normal vector which is outward with
respect to $\Omega _{f},$ and so inwards with respect to $\Omega _{s}.$ The
two dimensional vector $n_{j},$ $1\leq j\leq K,$ will denote the exterior
normal vector with respect to polygonal face $\Gamma _{j}.$ With $\{\Omega
_{s},\Omega _{f}\}$ as given, the PDE system under consideration is as
follows:

\begin{equation}
\left\{ 
\begin{array}{l}
u_{t}-\Delta u=0\text{ \ \ \ in \ }(0,T)\times \Omega _{f} \\ 
u|_{\Gamma _{f}}=0\text{ \ \ \ on \ }(0,T)\times \Gamma _{f};%
\end{array}%
\right.   \label{2a}
\end{equation}

\begin{equation}
\left\{ 
\begin{array}{l}
\frac{\partial ^{2}}{\partial t^{2}}h_{j}-\text{div}\sigma _{\Gamma _{s}}(h_{j})+h_{j}=\nu \cdot \sigma (w)|_{\Gamma _{j}}-\frac{\partial u}{\partial \nu }%
|_{\Gamma _{j}}\text{ \ \ \ on \ }(0,T)\times \Gamma _{j},\text{ \ \ for  }%
1\leq l\leq K \\ 
h_{j}|_{\partial \Gamma _{j}\cap \partial \Gamma _{l}}=h_{l}|_{\partial
\Gamma _{j}\cap \partial \Gamma _{l}}\text{ on \ }(0,T)\times (\partial
\Gamma _{j}\cap \partial \Gamma _{l})\text{, for all }1\leq l\leq K\text{
such that }\partial \Gamma _{j}\cap \partial \Gamma _{l}\neq \emptyset  \\ 
\left. n_{j}\cdot \sigma _{\Gamma _{s}}(h_{j})\right\vert _{\partial \Gamma
_{j}\cap \partial \Gamma _{l}}=-\left. n_{l}\cdot \sigma _{\Gamma
_{s}}(h_{l})\right\vert _{\partial \Gamma _{j}\cap \partial \Gamma _{l}}%
\text{on \ }(0,T)\times (\partial \Gamma _{j}\cap \partial \Gamma _{l})\\
\text{for all }1\leq l\leq K\text{ such that }\partial \Gamma _{j}\cap \partial
\Gamma _{l}\neq \emptyset .\text{\ }%
\end{array}%
\right.   \label{2.5b}
\end{equation}%
\begin{equation}
\left\{ 
\begin{array}{l}
w_{tt}-\text{div}\sigma (w)=0\text{ \ \ \ on \ }(0,T)\times \Omega _{s} \\ 
w_{t}|_{\Gamma _{j}}=\frac{\partial }{\partial t}h_{j}=u|_{\Gamma _{j}}\text{
\ \ \ on \ }(0,T)\times \Gamma _{j}\text{, \ for }j=1,...,K%
\end{array}%
\right.   \label{2d}
\end{equation}%
\begin{equation}
\lbrack u(0),h_{1}(0),\frac{\partial }{\partial t}h_{1}(0),...,h_{K}(0),%
\frac{\partial }{\partial t}%
h_{K}(0),w(0),w_{t}(0)]=[u_{0},h_{01},h_{11},...,h_{0K},h_{1K},w_{0},w_{1}].
\label{IC}
\end{equation}%
Here, the stress tensors $\sigma $ and $\sigma _{\Gamma _{s}}$ constitute
Lam$\acute{e}$ systems of elasticity on their respective ``thick" and ``thin" layers.
Namely,\newline \\
\textbf{i)} For function $v$ in $\Omega _{s},$  
\[
\sigma (v)=2\mu \epsilon (v)+\lambda \lbrack I_{3}\cdot \epsilon (v)]I_{3},
\]%
where strain tensor $\epsilon (\cdot )$ is given by 
\[
\epsilon _{ij}(v)=\frac{1}{2}\left( \frac{\partial v_{j}}{\partial x_{i}}+%
\frac{\partial v_{i}}{\partial x_{j}}\right) ,\text{ \ \ }1\leq i,j\leq 3;
\]%
\textbf{ii)} Likewise, for function $g$ on polygon $\Gamma _{j},$%
\[
\sigma _{\Gamma _{s}}(g)=2\mu _{\Gamma _{s}}\epsilon _{\Gamma
_{s}}(g)+\lambda _{\Gamma _{s}}[I_{2}\cdot \epsilon _{\Gamma _{s}}(g)]I_{2}
\]%
with%
\[
(\epsilon _{\Gamma _{s}})_{ik}(g)=\frac{1}{2}\left( \frac{\partial g_{k}}{%
\partial x_{i}}+\frac{\partial g_{i}}{\partial x_{k}}\right) ,\text{ \ \ }%
1\leq i,k\leq 2.
\]%
where $\{ \mu,\lambda \}$, $\{ \mu_{\Gamma _{s}},\lambda_{\Gamma _{s}} \}$ are respective Lam$\acute{e}$ parameters. We will consider said multi-layered-heat PDE system with initial data (\ref%
{IC}) drawn from the natural finite energy space $\mathbf{H}$, defined as: 
\begin{equation}
\begin{array}{l}
\mathbf{H}=\Big\{[u_{0},h_{01},h_{11},...,h_{0K},h_{1K},w_{0},w_{1}]\in
L^{2}(\Omega _{f})\times H^{1}(\Gamma _{1})\times L^{2}(\Gamma _{1})\times
... \\ 
\text{ \ \ \ \ \ \ \ \ \ \ \ }\times H^{1}(\Gamma _{K})\times L^{2}(\Gamma
_{K})\times H^{1}(\Omega _{s})\times L^{2}(\Omega _{s})\text{, \ such that
for each }1\leq j\leq K: \\ 
\text{\ \ \ \ \ \ \ \ \ \ (i) }w_{0}|_{\Gamma _{j}}=h_{0j}\text{; } \\ 
\text{ \ \ \ \ \ \ \ \ } \text{(ii) }h_{0j}|_{\partial \Gamma _{j}\cap
\partial \Gamma _{l}}=h_{0l}|_{\partial \Gamma _{j}\cap \partial \Gamma _{l}}%
\text{ on \ }\partial \Gamma _{j}\cap \partial \Gamma _{l}\text{, for all }%
1\leq l\leq K\text{ such that }\partial \Gamma _{j}\cap \partial \Gamma
_{l}\neq \emptyset \Big\}. %
\end{array}
\label{H}
\end{equation}%
Here, $\mathbf{H}$ is a Hilbert space with the inner product%
\begin{eqnarray}
(\Phi _{0},\widetilde{\Phi }_{0})_{\mathbf{H}} &=&(u_{0},\widetilde{u}%
_{0})_{\Omega _{f}}+\sum\limits_{j=1}^{K}(\sigma _{\Gamma
_{s}}(h_{0j}),\epsilon _{\Gamma _{s}}(\widetilde{h}_{0j}))_{\Gamma
_{j}}+\sum\limits_{j=1}^{K}(h_{0j},\widetilde{h}_{0j})_{\Gamma _{j}} 
\nonumber \\
&&+\sum\limits_{j=1}^{K}(h_{1j},\widetilde{h}_{1j})_{\Gamma _{j}}+(
\sigma(w_{0}),\epsilon( \widetilde{w}_{0}))_{\Omega _{s}}+(w_{1},\widetilde{w}%
_{1})_{_{\Omega _{s}}},  \label{Hilbert}
\end{eqnarray}%
where%
\begin{equation}
\Phi _{0}=\left[ u_{0},h_{01},h_{11},...,h_{0K},h_{1K},w_{0},w_{1}\right]
\in \mathbf{H}\text{; \ }\widetilde{\Phi }_{0}=\left[ \widetilde{u}_{0},%
\widetilde{h}_{01},\widetilde{h}_{11},...,\widetilde{h}_{0K},\widetilde{h}%
_{1K},\widetilde{w}_{0},\widetilde{w}_{1}\right] \in \mathbf{H}.
\label{stat}
\end{equation}

\subsection{Main Objective and Literature}

\hspace{0.5cm} The PDE model (\ref{2a})-(\ref%
{IC}) is one amongst a class of coupled PDE systems which have been derived, so as to mathematically describe the interaction between viscous blood flow and the multi-layered vessels through which such flow is transported within a given mammalian species; see e.g., \cite{multi-layered, BorisSimplifiedFSI}. (See also the following references which generally deal with the mathematical and/or modeling analysis of coupled (single-layered) structure-fluid PDE systems \cite{A-T, AvalosTriggiani09, Barbu, FSIforBIO, SunnyStentsSIAM, Chambolle, Courtand, Du, GazzolaReview, lions1969quelques, RichterBook}.)

In this work, we consider the strong stability problem; namely that of ascertaining that the thick elastic-thin elastic-thermal solution components tend asymptotically to the zero state, for given finite energy data in $\mathbf{H};$ see Theorem \ref{st} below. In particular, we investigate whether the dissipation which emanates only from the thermal component of the coupled system (\ref{2a})-(\ref%
{IC}), suffices to strongly stabilize the elastic dynamics, notwithstanding the fact that the three distinct PDE components each evolve on their own respective geometries.

We emphasize here that the domain of the associated thick Lam$\acute{e}$-thin Lam$\acute{e}$-heat semigroup generator is \underline{not} compactly embedded into the finite energy space $\mathbf{H}$-- see (\ref{4a}) and (A.i)-(A.iv)  below-- consequently, a conclusion of strong stability here will not follow from classic PDE control arguments, for which it is essentially sufficient (given an underlying compactness of the resolvent of the associated semigroup generator) to establish weak stability; see e.g., \cite{Levan}. With reference to such means, the fundamental example in the literature is the strong stability problem for the boundary damped wave equation on a bounded domain (see \cite{Q-R, trigg, Slem, Lagnese}). Again, what allows for said approach is the fact that the semigroup generator of the boundary damped wave equation has domain which is compactly embedded into $H^1\times L^2.$ In the present situation, this avenue of approach is not available. 

An analogous result of asymptotic decay was obtained in \cite{jde} for a canonical ``thick" wave - ``thin" wave - heat PDE model. Likewise, as in the present situation, the associated multi-layered structure - heat semigroup generator in \cite{jde} does not have compact resolvent. However, in \cite{jde}, with a view of ultimately invoking the wellknown spectral criterion in \cite{A-B} for strong stability, the authors were compelled to invoke a PDE multiplier method (in the frequency domain) so as to derive a wave identity for the thick wave PDE component. (Such wave identities for uncoupled dynamics are of course instrumental in establishing uniform stabilization; see \cite{chen, trigg, Lagnese}.) So in some sense, the partial multiplier approach to strong decay in \cite{jde} resembles that of \cite{Lagnese} for said boundary damped wave equation.

By contrast, we intend in the present work to pursue an approach which eschews the need for deriving analogous energy identities for the thick Lam$\acute{e}$ solution component of the multilayered-thermal system. Certainly, such Lam$\acute{e}$ energy identities exist (although of course they are a bit cumbersome) and have been used in the context of PDE boundary stabilization problems--see; e.g., \cite{Horn} and \cite{AvalosTriggiani09}-- however, since the present issue is one of strong, and not uniform decay, it would seem preferable to find a ``softer" functional approach--somewhat in the spirit of the aforesaid works on boundary damped wave equation strong decay-- at least to the extent possible.

To this end, our strong stability proof here is predicated upon ultimately invoking the \underline{resolvent} criterion in \cite{Tomilov}; see Theorem \ref{C-T} below. Essentially, in order to infer strong decay of finite energy solutions of (\ref{2a})-(\ref%
{IC}), we will show below that the associated thick Lam$\acute{e}$ - thin Lam$\acute{e}$ - heat semigroup generator has (noncompact) resolvent which ``almost everywhere" obeys a certain strong limit with respect to parameter values in the right half complex plane. In order to avail ourselves of Theorem 4, we must as a preliminary step establish that the multilayered structure heat semigroup (besides being a contraction) is also completely non-unitary (c.n.u). In this step, we will need to appeal  to the relatively recent global uniqueness (Holmgren's-type) result for Lam$\acute{e}$ systems of elasticity; see \cite{Eller-Tound}. Moreover, we will need to recall higher regularity results for uncoupled three dimensional Lam$\acute{e}$ systems of elasticity on polyhedra; see \cite{Grisvard}. 

We intend, as future work, to investigate uniform decay properties of (\ref{2a})-(\ref%
{IC})-- probably taking as our point of departure the canonical multilayered structure - heat system in \cite{jde}. Accordingly, we should mention those results of exponential and polynomial decay in the literature for single-layered structure - parabolic PDE models, \cite{RauchZhangZuazua, Duyckaerts, AvalosTriggiani2, Avalos-Trigg, AvalosLasieckaTriggiani16, AG1, A-P2, Las-Leb}.

\subsection{Notation}

For the remainder of the text norms $||\cdot ||_{D}$ are taken to be $%
L^{2}(D)$ for the domain $D$. Inner products in $L^{2}(D)$ is written $%
(\cdot ,\cdot )$, while inner products $L_{2}(\partial D)$ are written $%
\langle \cdot ,\cdot \rangle $. The space $H^{s}(D)$ will denote the Sobolev
space of order $s$, defined on a domain $D$, and $H_{0}^{s}(D)$ denotes the
closure of $C_{0}^{\infty }(D)$ in the $H^{s}(D)$ norm which we denote by $%
\Vert \cdot \Vert _{H^{s}(D)}$ or $\Vert \cdot \Vert _{s,D}$. We make use of
the standard notation for the trace of functions defined on a Lipschitz
domain $D$, i.e. for a scalar function $\phi \in H^{1}(D)$, we denote $%
\gamma (w)$ to be the trace mapping from $H^{1}(D)$ to $H^{1/2}(\partial D)$%
. We will also denote pertinent duality pairings as $(\cdot ,\cdot
)_{X\times X^{\prime }}$.

\section{Preliminaries}

With respect to the above setting, the PDE system given in (\ref{2a})-(\ref%
{IC}) may be associated with an abstract ODE in Hilbert space $\mathbf{H}$.
To wit, the operator $\mathbf{A}:D(\mathbf{A})\subset \mathbf{H}\rightarrow 
\mathbf{H}$ be defined by%
\begin{equation}
\mathbf{A}=\left[ 
\begin{array}{cccccccc}
\Delta  & 0 & 0 & 0 & 0 & 0 & 0 & 0 \\ 
0 & 0 & I & \cdots  & 0 & 0 & 0 & 0 \\ 
-\frac{\partial }{\partial \nu }|_{\Gamma _{1}} & (\text{div}\sigma _{\Gamma
_{s}}-I) & 0 & \cdots  & 0 & 0 & \nu \cdot \sigma (\cdot )|_{\Gamma _{1}} & 0
\\ 
\vdots  & \vdots  & \vdots  & \cdots  & \vdots  & \vdots  & \vdots  & \vdots 
\\ 
0 & 0 & 0 & \cdots  & 0 & I & 0 & 0 \\ 
-\frac{\partial }{\partial \nu }|_{\Gamma _{K}} & 0 & 0 & \cdots  & (\text{%
div}\sigma _{\Gamma _{s}}-I) & 0 & \nu \cdot \sigma (\cdot )|_{\Gamma _{K}}
& 0 \\ 
0 & 0 & 0 & \cdots  & 0 & 0 & 0 & I \\ 
0 & 0 & 0 & \cdots  & 0 & 0 & \text{div}\sigma (\cdot ) & 0%
\end{array}%
\right] ;  \label{4a}
\end{equation}%
\begin{equation}
\begin{array}{l}
D(\mathbf{A})=\Big\{ \left[ u_{0},h_{01},h_{11},\ldots
,h_{0K},h_{1K},w_{0},w_{1}\right] \in \mathbf{H}:  \\ 
\text{ \ \ \textbf{(A.i)} }u_{0}\in H^{1}(\Omega _{f})\text{, }h_{1j}\in
H^{1}(\Gamma _{j})\text{ for }1\leq j\leq K\text{, }w_{1}\in H^{1}(\Omega
_{s})\text{;} \\ 
\text{ \ } \text{\textbf{(A.ii)} \textbf{(a)} }\Delta u_{0}\in L^{2}(\Omega
_{f})\text{, \textbf{(b)} div}\sigma (w_{0})\in L^{2}(\Omega _{s}), \\
\text{\ \ \ \ \ \ \ \ \ \ \  \textbf{(c)} div}%
\sigma _{\Gamma _{s}}(h_{0j})+\nu \cdot \sigma (w_{0})|_{\Gamma _{j}}-\frac{%
\partial u_{0}}{\partial \nu }|_{\Gamma _{j}}\in L^{2}(\Gamma _{j})\text{ \
for \ }1\leq j\leq K\text{;} \\ 
\text{ \ }\left. \text{\textbf{(A.iii)} }u_{0}|_{\Gamma _{f}}=0,\ \
u_{0}|_{\Gamma _{j}}=h_{1j}=w_{1}|_{\Gamma _{j}},\ \text{for }1\leq j\leq K%
\text{;}\right.  \\ 
\text{ \ }\left. \text{\textbf{(A.iv)} For }1\leq j\leq K\text{: }\right. 
\\ 
\text{ \ \ \ \ \ \ \ \ \ \ \textbf{(a)} }h_{1j}|_{\partial \Gamma _{j}\cap \partial \Gamma
_{l}}=h_{1l}|_{\partial \Gamma _{j}\cap \partial \Gamma _{l}}\text{ on \ }%
\partial \Gamma _{j}\cap \partial \Gamma _{l}\text{, for all }1\leq l\leq K%
\text{ such that }\partial \Gamma _{j}\cap \partial \Gamma _{l}\neq
\emptyset ; \\ 
\text{ \ \ \ \ \ \ \ \ \  }\left. \text{\textbf{(b)} } n_{j}\cdot \sigma _{\Gamma
_{s}}(h_{0j})\right\vert _{\partial \Gamma _{j}\cap \partial \Gamma
_{l}}=-\left. n_{l}\cdot \sigma _{\Gamma _{s}}(h_{0l})\right\vert _{\partial
\Gamma _{j}\cap \partial \Gamma _{l}}\text{, for all }1\leq l\leq K\text{
such that }\partial \Gamma _{j}\cap \partial \Gamma _{l}\neq \emptyset
\Big\} .%
\end{array}
\label{dom}
\end{equation}%
\\

\noindent With this matrix, if $\Phi (t)=\left[ u(t),h_{1}(t),\frac{\partial }{%
\partial t}h_{1}(t),...,h_{K}(t),\frac{\partial }{\partial t}%
h_{K}(t),w(t),w_{t}(t)\right] ,$ and\\  $\Phi
_{0}=[u_{0},h_{01},h_{11},...,h_{0K},h_{1K},w_{0},w_{1}]$ then the solution
of (\ref{2a})-(\ref{IC}) may be written simply as%
\begin{equation}
\frac{d}{dt}\Phi (t)=\mathbf{A}\Phi (t)\text{; \ }\Phi (0)=\Phi _{0}.
\label{ODE}
\end{equation}
Proceeding along similar lines of approach as in \cite{jde}, one can obtain the
following result of well posedness:

\begin{theorem}
\label{well}The operator $\mathbf{A}:D(\mathbf{A})\subset \mathbf{H}%
\rightarrow \mathbf{H}$, defined in (\ref{4a})-(\ref{dom}), generates a $%
C_{0}$-semigroup of contractions on $\mathbf{H}$. Consequently, the solution\\ 
$\Phi (t)=\left[ u(t),h_{1}(t),\frac{\partial }{%
\partial t}h_{1}(t),...,h_{K}(t),\frac{\partial }{\partial t}%
h_{K}(t),w(t),w_{t}(t)\right] $ of (\ref{2a})-(\ref{IC}), or
equivalently (\ref{ODE}), is given by 
\[
\Phi (t)=e^{\mathbf{A}t}\Phi _{0}\in C([0,T];\mathbf{H})\text{,}
\]%
where $\Phi _{0}=\left[ u_{0},h_{01},h_{11},...,h_{0K},h_{1K},w_{0},w_{1}%
\right] \in \mathbf{H}$.
\end{theorem}
In fact, the main result of this manuscript is to show that the solution to
the system (\ref{2a})-(\ref{IC}) has a decay to the zero state. To prove
this, we firstly need to give the following dissipation estimate from which the
decay arises: 

\begin{proposition}
\label{diss} The solution of  (\ref{2a})-(\ref{IC}), or equivalently (\ref%
{ODE}), satisfies the following relation for $0\leq t_{0}\leq t\leq \infty :$%
\begin{equation}
2\int\limits_{t_0}^{t}\left\Vert \nabla u(\tau )\right\Vert _{\Omega
_{f}}^{2}d\tau +\left\Vert \Phi (t)\right\Vert _{\mathbf{H}}^{2}=\left\Vert
\Phi (t_{0})\right\Vert _{\mathbf{H}}^{2}  \label{onbir}
\end{equation}
\end{proposition}

\begin{proof}
With respect to the right hand side of (\ref{ODE}) we have, upon taking the $\mathbf{H}$-inner product with respect to $\Phi$, and then an
integration over $(t_0,t),$%
\begin{eqnarray*}
\int\limits_{t_{0}}^{t}\left( \mathbf{A}\Phi (\tau ),\Phi (\tau )\right) _{%
\mathbf{H}}d\tau  &=&\int\limits_{t_{0}}^{t}\left( \Delta u,u\right)
_{\Omega _{f}}d\tau +\sum\limits_{j=1}^{K}\int\limits_{t_{0}}^{t}\left[
(\sigma _{\Gamma _{s}}(\frac{\partial }{\partial t}h_{j}),\epsilon _{\Gamma
_{s}}(h_{j}))_{\Gamma _{j}}+(\frac{\partial }{\partial t}h_{j},h_{j})_{%
\Gamma _{j}}\right] d\tau  \\
&&+\sum\limits_{j=1}^{K}\int\limits_{t_{0}}^{t}\left[ (\text{div}\sigma
_{\Gamma _{s}}(h_{j}),\frac{\partial }{\partial t}h_{j})_{\Gamma
_{j}}-(h_{j},\frac{\partial }{\partial t}h_{j})_{\Gamma _{j}}\right] d\tau 
\\
&&+\sum\limits_{j=1}^{K}\int\limits_{t_{0}}^{t}\left\langle \nu \cdot
\sigma (w)-\frac{\partial u}{\partial \nu },\frac{\partial }{\partial t}h_{j}\right\rangle _{\Gamma
_{j}}d\tau +\int\limits_{t_{0}}^{t}\left( \sigma (w_{t}),\epsilon
(w)\right) _{\Omega _{s}}d\tau  \\
&&+\int\limits_{t_{0}}^{t}\left( \text{div}\sigma (w),\epsilon
(w_{t})\right) _{\Omega _{s}}d\tau  \\
&=&-\int\limits_{t_{0}}^{t}\left( \nabla u,\nabla u\right) _{\Omega
_{f}}d\tau +\sum\limits_{j=1}^{K}\int\limits_{t_{0}}^{t}\left[ (\sigma
_{\Gamma _{s}}(\frac{\partial }{\partial t}h_{j}),\epsilon _{\Gamma
_{s}}(h_{j}))_{\Gamma _{j}}+(\frac{\partial }{\partial t}h_{j},h_{j})_{%
\Gamma _{j}}\right] d\tau \\
 &&-\sum\limits_{j=1}^{K}\int\limits_{t_{0}}^{t}\left[ \overline{(\sigma
_{\Gamma _{s}}(\frac{\partial }{\partial t}h_{j}),\epsilon _{\Gamma
_{s}}(h_{j}))_{\Gamma _{j}}}+\overline{(\frac{\partial }{\partial t}%
h_{j},h_{j})_{\Gamma _{j}}}\right] d\tau \\
&&+\int\limits_{t_{0}}^{t}\left( \sigma (w_{t}),\epsilon (w)\right)
_{\Omega _{s}}d\tau -\int\limits_{t_{0}}^{t}\overline{\left( \sigma
(w_{t}),\epsilon (w)\right) _{\Omega _{s}}}d\tau \\
 \end{eqnarray*}
\begin{eqnarray*}
&&+\int\limits_{t_{0}}^{t}\left\langle \frac{\partial u}{\partial \nu }%
,u\right\rangle _{\partial \Omega _{f}}d\tau
+\sum\limits_{j=1}^{K}\int\limits_{t_{0}}^{t}\left\langle \nu \cdot \sigma
(w)-\frac{\partial u}{\partial \nu },\frac{\partial }{\partial t}h_{j}\right\rangle _{\Gamma _{j}}d\tau 
\\
&&-\int\limits_{t_{0}}^{t}\left\langle \nu \cdot \sigma
(w),w_{t}\right\rangle _{\Gamma _{s}}d\tau
+\sum\limits_{j=1}^{K}\int\limits_{t_{0}}^{t}\left\langle n_{j}\cdot \sigma
_{\Gamma _{s}}(h_{j}),\frac{\partial }{\partial t}h_{j}\right\rangle
_{\Gamma _{j}}d\tau .
\end{eqnarray*}%
Invoking the BCs in (A.iii) and (A.iv), imposed on the
structure-structure-heat variables, we then have%
\begin{equation}
\int\limits_{t_{0}}^{t}\left( \mathbf{A}\Phi (\tau ),\Phi (\tau )\right) _{%
\mathbf{H}}d\tau =-\int\limits_{t_{0}}^{t}\left\Vert \nabla u\right\Vert
_{\Omega _{f}}^{2}d\tau +2i\text{Im}\sum\limits_{j=1}^{K}\int%
\limits_{t_{0}}^{t}\left[ (\sigma _{\Gamma _{s}}(\frac{\partial }{\partial t}%
h_{j}),\epsilon _{\Gamma _{s}}(h_{j}))_{\Gamma _{j}}+(\frac{\partial }{%
\partial t}h_{j},h_{j})_{\Gamma _{j}}\right] d\tau.   \label{oniki}
\end{equation}%
Applying this relation to the RHS of the relation, 
\[
\int\limits_{t_{0}}^{t}\left( \frac{d}{d\tau }\Phi (\tau ),\Phi (\tau
)\right) _{\mathbf{H}}d\tau =\int\limits_{t_{0}}^{t}\left( \mathbf{A}\Phi
(\tau ),\Phi (\tau )\right) _{\mathbf{H}}d\tau 
\]%
we obtain the desired estimate.
\end{proof}

\section{Main Result: Strong Stability via Resolvent Criterion}

This section is devoted to addressing the issue of asymptotic behavior of the solution whose existence - uniqueness is guaranteed by Theorem \ref{well}. In this regard, we show that the system given in (\ref%
{2a})-(\ref{IC}) is strongly stable in the finite energy space $\mathbf{H}$. Our main  result is as follows:

\begin{theorem}
\label{st} Finite energy solutions of the multilayered-heat PDE system (\ref%
{2a})-(\ref{IC}), or equivalently (\ref{ODE}) decay strongly to zero. That
is to say, the solution of (\ref{2a})-(\ref{IC}), with initial data $\Phi
_{0}\in $ $\mathbf{H,}$ satisfies%
\[
\lim_{t\rightarrow 0}\left\Vert \Phi (t)\right\Vert _{\mathbf{H}}=0.
\]
\end{theorem}
To prove Theorem \ref{st}, in contrast to the approach taken in \cite{jde} which is geared to invoke the
wellknown spectral criteria for stability in \cite{A-B}, we will adopt
instead a resolvent-based methodology. In particular, we will ultimately
appeal to the following theorem (see \cite[Theorem 8.7]{CT} or
\cite[pp.75-76]{Tomilov}, see also \cite{Levan-2}.) 

\begin{theorem}
\label{C-T} Let $A$ generate a $C_{0}-$semigroup of completely, non-unitary
contractions on a Hilbert space $H$. If there exists a dense set $M\subset H$
such that%
\begin{equation}
\lim_{\alpha \rightarrow 0^{+}}\sqrt{\alpha }R(\alpha +i\beta ;A)x=0
\label{oniki-5}
\end{equation}%
for every $x\in M$ and almost every $\beta \in 
\mathbb{R}
,$ then the semigroup is strongly stable.
\end{theorem}

\begin{remark}
We recall that an operator $L\in \mathcal{L}(H)$ is completely non-unitary
(c.n.u) if the trivial subspace is the only subspace of $H$ which reduces $L$
to a unitary operator. (See, e.g., \cite{Levan})
\end{remark}

\begin{lemma}
\label{cnu} The given elastic-elastic-heat semigroup $\{e^{\mathbf{A}t}\}_{t\geq 0}$ 
is completely non-unitary (c.n.u).
\end{lemma}

\begin{proof}
With reference to problem (\ref{2a})-(\ref{IC}), assume that initial data $%
\Phi _{0}$  is drawn from an invariant subspace $W\subset \mathbf{H}$ on
which the operator family $\{e^{\mathbf{A}t}\}_{t\geq 0}$ is unitary. Then from the expression (\ref{ODE}) and the dissipative relation (%
\ref{onbir}) of Proposition \ref{diss}, we have that the heat component of (%
\ref{2a})-(\ref{IC}) satisfies%
\begin{equation}
u=0\text{ \ \ \ \ on \ \ }(0,T)\times \Omega _{f}.  \label{onuc}
\end{equation}%
Consequently from the BC's in (\ref{2d}) we have%
\begin{equation}
w_{t}|_{\Gamma _{s}}=0\text{ \ \ on \ }(0,T)\times \Gamma _{s},\text{\ }
\label{14.a}
\end{equation}%
\begin{equation}
\frac{\partial }{\partial t}h_{j}=0\text{ \ \ in \ }(0,T)\times \Gamma _{j}%
\text{ \ \ for }j=1,...,K.\text{\ }  \label{14.b}
\end{equation}%
Differentiating the thin elastic equations in (\ref{2.5b}), and subsequently
invoking (\ref{onuc}) and (\ref{14.b}) we have in turn%
\begin{equation}
\nu \cdot \sigma (w_{t})|_{\Gamma _{s}}=0\text{ \ \ on \ }(0,T)\times \Gamma
_{s}.  \label{15}
\end{equation}%
If we now make the change of variable $v\equiv w_{t}$, then from (\ref{2d}),
(\ref{14.a}) and (\ref{15}), we have that $v$ satisfies the overdetermined
Lam$\acute{e}$ system%
\begin{equation}
\left\{ 
\begin{array}{c}
v_{tt}-\text{div}\sigma (v)=0\text{ \ \ on \ }(0,T)\times \Omega _{s} \\ 
v|_{\Gamma _{s}}=\nu \cdot \sigma (v)|_{\Gamma _{s}}=0\text{ \ \ on \ }%
(0,T)\times \Gamma _{s}.%
\end{array}%
\right.   \label{16}
\end{equation}%
Consequently, from the Holmgren's type result in \cite{Eller-Tound} \footnote{In \cite{Eller-Tound}, the geometry was assumed to be $C^1$. However the details of proof apply readily to piecewise $C^1-$domains. Indeed, Holmgren's uniqueness will hold for Lam$\acute{e}$ systems on Lipschitz domains; see \cite{Eller-3}.} (see also \cite{Eller-et-al}) we get for $T>0$ sufficiently large,%
\begin{equation}
v=0\text{ \ \ OR \ \ }w_{t}=0\text{ \ in \ }(0,T)\times \Omega _{s},\text{\ }
\label{17}
\end{equation}%
and so%
\begin{equation}
w=\text{constant \ \ in  }(0,T)\times \Omega _{s}.\text{\ \ }  \label{18}
\end{equation}%
From the compatibility condition between thick and thin elastic displacements in (\ref{H}), we then have from (\ref{18}) that $$h_{j}=constant \text{ \ in \ }(0,T)\times
\Gamma _{j},\text{ \ for }1\leq j\leq K.$$
Applying this consequence to the thin elastic equation in (\ref{2.5b}) and further invoking (\ref{onuc}), (\ref{14.b}) and (\ref{18}) we get 
\begin{equation}
 h_{j}=0\text{ \ in \ }(0,T)\times
\Gamma _{j},\text{ \ for }1\leq j\leq K.  \label{19}
\end{equation}%
Finally, (\ref{18}), (\ref{19}) and said compatibility condition between
thick and thin elastic displacements, imposed in the natural energy space $%
\mathbf{H,}$ yield %
\begin{equation}
w=0\text{\ \ on \ }(0,T)\times \Omega _{s}\text{.}  \label{20}
\end{equation}%
To conclude, (\ref{onuc}), (\ref{14.a}), (\ref{14.b}), (\ref{19}) and (\ref%
{20}) yield on $(0,T)$
\[
e^{\mathbf{A}t}\Phi _{0}=0,\text{ \ }\Phi _{0}\in W
\]%
so necessarily $W=\{0\}.$ This finishes the proof of Lemma \ref{cnu}.
\end{proof}
\\

\noindent Before embarking on our proof of strong stability, we recall the following
regularity result for the homogeneous boundary value problem involving the
Lam$\acute{e}$ system of elasticity on polyhedron $\Omega _{s}.$ (We are not aiming
here for great generality.)

\begin{proposition}
\label{reg} Suppose $z\in H^{1}(\Omega _{s})$ satisfies the BVP%
\[
\left\{ 
\begin{array}{c}
-\text{div}\sigma (z)=f\in L^{2}(\Omega _{s}) \\ 
z|_{\Gamma _{s}}=0\text{ \ \ on \ }\Gamma _{s}.%
\end{array}%
\right. 
\]%
Then $z$ has the higher regularity%
\begin{equation}
\left\Vert z\right\Vert _{H^{\frac{3}{2}+\epsilon }(\Omega
_{s})}+\sum\limits_{j=1}^{K}\left\Vert \nu \cdot \sigma (z)|_{\Gamma
_{j}}\right\Vert _{H^{\epsilon }(\Gamma _{j})}\leq C\left\Vert f\right\Vert
_{\Omega _{s}}.  \label{22}
\end{equation}
\end{proposition}

\begin{proof}
The fact that $z\in H^{\frac{3}{2}+\epsilon }(\Omega _{s})$ comes
immediately from \cite[Theorem 4.5.1, p. 140, see also the remark on p. 149]{Grisvard}. (See also Theorem 4.5
of \cite{Nicaise}.) Moreover, given a point on boundary face $\Gamma _{j},$ let
unit (tangent) vectors $\tau _{1,}\tau _{2}$ --we neglect the index $j$ here--
be such that $\{\nu (x),\tau _{1}(x),\tau _{2}(x)\}$ constitutes an
orthonormal basis of $%
\mathbb{R}
^{3}.$ Therewith, one can compute outright --see e.g., Proposition A.1 (iii)
of \cite{Avalos-Trigg}-- the expression for $j=1,2...,K:$%
\begin{eqnarray}
\nu \cdot \sigma (z)|_{\Gamma _{j}} &=&\lambda \left[ \frac{\partial z}{%
\partial \nu }\cdot \nu +\frac{\partial z}{\partial \tau _{1}}\cdot \tau
_{1}+\frac{\partial z}{\partial \tau _{2}}\cdot \tau _{2}\right] \nu  
\nonumber \\
&&+2\mu \frac{\partial z}{\partial \nu }+\mu \left[ (\frac{\partial w}{%
\partial \tau _{2}}\cdot \nu )-(\frac{\partial w}{\partial \nu }\cdot \tau
_{2})\right] \tau _{2}  \nonumber \\
&&+\mu \left[ (\frac{\partial w}{\partial \tau _{1}}\cdot \nu )-(\frac{%
\partial w}{\partial \nu }\cdot \tau _{1})\right] \tau _{1}.\text{ \ \ }
\label{23}
\end{eqnarray}%
To deal with RHS, we recall the known bounded Sobolev trace maps for a
polyhedron (see e.g., Theorem 6.9 (i), page 43 of \cite{Dauge-et-al}):%
\begin{eqnarray}
z &\in &H^{\frac{3}{2}+\epsilon }(\Omega _{s})\rightarrow z|_{\Gamma
_{j}}\in H^{1+\epsilon }(\Gamma _{j}),\text{ \ }\ j=1,2,...,K  \nonumber \\
z &\in &H^{\frac{3}{2}+\epsilon }(\Omega _{s})\rightarrow \frac{\partial z}{%
\partial \nu }|_{\Gamma _{j}}\in H^{\epsilon }(\Gamma _{j}),\text{ \ }%
j=1,2,...,K.  \label{24}
\end{eqnarray}%
Applying these maps to RHS of (\ref{23}), and invoking said continuous
mapping $f\rightarrow z$ in (\ref{22}) now completes the proof.
\end{proof}
\newpage
\noindent \textbf{Proof of Theorem \ref{st} - Strong Stability}

\bigskip \noindent Having established that the multi-layered structure-heat PDE
contraction semigroup $\{e^{\mathbf{A}t}\}_{t\geq 0}$ is c.n.u by Lemma \ref%
{cnu}, we can make use of the resolvent criterion given in Theorem \ref{C-T}. To this
end, we define the operator $A_{D}:D(A_{D})\subset L^{2}(\Omega
_{s})\rightarrow L^{2}(\Omega _{s})$ via 
\begin{equation}
\left\{ 
\begin{array}{c}
A_{D}f=-\text{div}\sigma (f) \\ 
D(A_{D})=\{f\in H_{0}^{1}(\Omega _{s}):-\text{div}\sigma (f)\in L^{2}(\Omega
_{s})\}.%
\end{array}%
\right.   \label{25}
\end{equation}%
With respect to this self-adjoint operator with compact inverse, we denote%
\begin{equation}
S\equiv \{\beta \in 
\mathbb{R}
:\beta ^{2}\text{ is an eigenvalue of }A_{D}:D(A_{D})\subset L^{2}(\Omega
_{s})\rightarrow L^{2}(\Omega _{s})\}. \label{25.5}
\end{equation}%
In order to invoke the resolvent criterion in Theorem \ref{C-T}, we will
establish that the thick elastic-thin elastic-heat generator obeys the
strong limit (\ref{oniki-5}) for all $\beta \in 
\mathbb{R}
\backslash (S\cup\{0\}).$ To this end, with $\alpha >0$ and \underline{fixed} $\beta
\in 
\mathbb{R}
\backslash (S\cup\{0\}),$ we consider the resolvent equation%
\begin{equation}
\left[ (\alpha +i\beta )I-\mathbf{A}\right] \Phi =\Phi _{0}^{\ast }
\label{26}
\end{equation}%
where the solution $\Phi =\left[
u_{0},h_{01},h_{11},...,h_{0K},h_{1K},w_{0},w_{1}\right] \in D\mathbf{(A)}$
and the data \\ $\Phi _{0}^{\ast }=\left[ u_{0}^{\ast },h_{01}^{\ast
},h_{11}^{\ast },...,h_{0K}^{\ast },h_{1K}^{\ast },w_{0}^{\ast },w_{1}^{\ast
}\right] \in \mathbf{H.}$ From the definition of $D\mathbf{(A),}$ this
abstract equation can be written explicitly as 
\begin{equation}
(\alpha +i\beta )u_{0}-\Delta u_{0}=u_{0}^{\ast }\text{ \ \ in }\Omega _{f}%
\text{\ }  \label{27.a}
\end{equation}%
For $1\leq j\leq K:$%
\begin{equation}
\left\{ 
\begin{array}{c}
(\alpha +i\beta )h_{0j}-h_{1j}=h_{0j}^{\ast }\text{ \ \ in \ }\Gamma _{j} \\ 
(\alpha +i\beta )h_{1j}-\text{div}\sigma _{\Gamma _{s}}(h_{0j})+h_{0j}+\frac{%
\partial u_{0}}{\partial \nu }-\nu \cdot \sigma (w_{0})=h_{1j}^{\ast }\text{
\ \ in \ }\Gamma _{j}%
\end{array}%
\right.   \label{27.bc}
\end{equation}%
\begin{equation}
h_{0j}|_{\partial \Gamma _{j}\cap \partial \Gamma _{l}}=h_{0l}|_{\partial
\Gamma _{j}\cap \partial \Gamma _{l}}\text{ on \ }(\partial \Gamma _{j}\cap
\partial \Gamma _{l})\text{, for all }1\leq l\leq K\text{ such that }%
\partial \Gamma _{j}\cap \partial \Gamma _{l}\neq \emptyset   \label{27.d}
\end{equation}%
\begin{equation}
n_{j}\cdot \sigma _{\Gamma _{s}}(h_{0j})|_{\partial \Gamma _{j}\cap \partial
\Gamma _{l}}=-n_{l}\cdot \sigma _{\Gamma _{s}}(h_{0l})|_{\partial \Gamma
_{j}\cap \partial \Gamma _{l}}\text{on }(\partial \Gamma _{j}\cap \partial
\Gamma _{l}) \text{, for all }1\leq l\leq K\text{ such that }\partial \Gamma
_{j}\cap \partial \Gamma _{l}\neq \emptyset   \label{27.e}
\end{equation}%
\begin{equation}
w_{1}=(\alpha +i\beta )w_{0}-w_{0}^{\ast }\text{ \ in }\Omega _{s}\text{\ }
\label{27.f}
\end{equation}%
\begin{equation}
-\beta ^{2}w_{0}-\text{div}\sigma (w_{0})=-(\alpha ^{2}+2i\alpha \beta
)w_{0}+(\alpha +i\beta )w_{0}^{\ast }+w_{1}^{\ast }\text{ \ \ in }\Omega _{s}
\label{27.g}
\end{equation}%
\begin{equation}
\lbrack (\alpha +i\beta )w_{0}-w_{0}^{\ast }]|_{\Gamma
_{j}}=h_{1j}=u_{0}|_{\Gamma _{j}}\text{ \ on \ }\Gamma _{j}.  \label{27.h}
\end{equation}%
Throughout, take $0<\alpha<M_0,$ for some positive constant $M_0,$ and we will give the proof in the following steps: \newline \newline
\textbf{STEP I }\textbf{(A static dissipation relation):} Taking the $(\cdot ,\cdot )$%
-inner product of both sides of (\ref{26}) with respect to $\Phi ,$ we have%
\[
\alpha \left\Vert \Phi \right\Vert _{\mathbf{H}}^{2}+i\beta \left\Vert \Phi
\right\Vert _{\mathbf{H}}^{2}-(\mathbf{A}\Phi ,\Phi )_{\mathbf{H}}=(\Phi _{0}^{\ast },\Phi )_{\mathbf{H}}.
\]
Proceeding as in the proof of Proposition \ref{diss}, we obtain%
\begin{eqnarray*}
\alpha \left\Vert \Phi \right\Vert _{\mathbf{H}}^{2}+i\beta \left\Vert \Phi
\right\Vert _{\mathbf{H}}^{2}+\left\Vert \nabla u_{0}\right\Vert _{\Omega
_{f}}^{2}&=&(\Phi
_{0}^{\ast },\Phi )_{\mathbf{H}}+2i\sum\limits_{j=1}^{K}\text{Im}(\sigma _{\Gamma
_{s}}(h_{1j}),\epsilon _{\Gamma _{s}}({h}_{0j}))_{\Gamma
_{j}}\\
&&+2i\sum\limits_{j=1}^{K}\text{Im}(h_{1j},{h}_{0j})_{\Gamma _{j}}+2i\text{Im}(
\sigma(w_{1}),\epsilon({w}_{0}))_{\Omega _{s}}\end{eqnarray*}%
or%
\begin{equation}
\alpha \left\Vert \Phi \right\Vert _{\mathbf{H}}^{2}+\left\Vert \nabla
u_{0}\right\Vert _{H}^{2}=\text{Re}(\Phi _{0}^{\ast },\Phi )_{\mathbf{H}}.
\label{28}
\end{equation}%
Invoking the boundary conditions (\ref{27.h}) and the Sobolev Trace Theorem,
we then have for $1\leq j\leq K$ 
\begin{equation}
\lbrack (\alpha +i\beta )w_{0}-w_{0}^{\ast }]|_{\Gamma
_{j}}=h_{1j}=u_{0}|_{\Gamma _{j}}=\sqrt {\mathcal{O}\left( \left\vert (\Phi
_{0}^{\ast },\Phi )_{\mathbf{H}}\right\vert \right)} \text{ \ on \ }\Gamma
_{j}.  \label{29}
\end{equation}%
\newline
\textbf{STEP II} \textbf{(An estimate for the thick elastic displacement):} We start
here by defining the ``Dirichlet" map $D:L^{2}(\Omega _{s})\rightarrow
L^{2}(\Omega _{s})$ via%
\begin{equation}
Dg=v\Longleftrightarrow \left\{ 
\begin{array}{c}
\text{div}\sigma (v)=0\text{ \ \ in \ }\Omega _{s} \\ 
v=g\text{ \ \ on \ }\Gamma _{s}%
\end{array}%
\right.   \label{29.5}
\end{equation}%
By the Lax-Milgram Theorem, and a subsequent integration by parts with respect to (\ref{29.5}), we have%
\begin{equation}
D\in \mathcal{L}(H^{\frac{1}{2}}(\Gamma _{s}),H^{1}(\Omega _{s})); \text{ \ \ \ \ \ } \nu \cdot \sigma(D(\cdot))\in \mathcal{L}(H^{\frac{1}{2}}(\Gamma _{s}),H^{-\frac{1}{2}}(\Gamma _{s})).
\label{29.6}
\end{equation}%
With this mapping, if we now let 
\begin{equation}
z=w_{0}-\frac{i}{\beta }D([\alpha w_{0}-u_{0}-w_{0}^{\ast }]|_{\Gamma _{s}}),
\label{29.7}
\end{equation}%
then from (\ref{27.g}), we have that $z$ solves the BVP:%
\[
\left\{ 
\begin{array}{c}
-\beta ^{2}z-\text{div}\sigma (z)=-i\beta D([\alpha w_{0}-u_{0}-w_{0}^{\ast
}]|_{\Gamma _{s}})-(\alpha ^{2}+2i\alpha \beta )w_{0}+(\alpha +i\beta )w_{0}^{\ast
}+w_{1}^{\ast }\text{ \ \ in }\Omega _{s} \\ 
 \\
z=0\text{ \ \ \ on \ }\Gamma _{s}.\text{\ }%
\end{array}%
\right. 
\]%
Since $\beta ^{2}$ is not an eigenvalue of $A_{D}$ (defined in (\ref{25})),
we then have%
\[
z=(\beta ^{2}-A_{D})^{-1} \Big[ i\beta D([\alpha w_{0}-u_{0}-w_{0}^{\ast }]|_{\Gamma _{s}}) +(\alpha ^{2}+2i\alpha \beta )w_{0}-(\alpha +i\beta )w_{0}^{\ast
}-w_{1}^{\ast }\Big] \text{ \ \ in }\Omega _{s}.\] 
Estimating RHS by means of (\ref{28}), (\ref{29.6}), (\ref{29}) and the Sobolev Trace
Theorem, we then have %
\begin{equation}
\left\Vert z\right\Vert _{\Omega _{s}}=\mathcal{O}\left( \sqrt{\left\vert (\Phi
_{0}^{\ast },\Phi )_{\mathbf{H}}\right\vert }+\left\Vert \Phi _{0}^{\ast
}\right\Vert _{\mathbf{H}}\right) .  \label{30}
\end{equation}%
In turn, by the higher regularity result in Proposition \ref{reg}, we have%
\begin{eqnarray*}
&&\left\Vert z\right\Vert _{H^{\frac{3}{2}+\epsilon }(\Omega
_{s})}+\sum\limits_{j=1}^{K}\left\Vert \nu \cdot \sigma (z)|_{\Gamma
_{j}}\right\Vert _{H^{\epsilon }(\Gamma _{j})} \\
&\leq &C\left\Vert \beta ^{2}z-i\beta D([\alpha w_{0}-u_{0}-w_{0}^{\ast }]|_{\Gamma
_{s}})-(\alpha ^{2}+2i\alpha \beta )w_{0}+(\alpha +i\beta )w_{0}^{\ast
}+w_{1}^{\ast }\right\Vert _{\Omega _{s}}.
\end{eqnarray*}%
Appealing to estimate (\ref{22}) and (\ref{30}) (and once more (\ref{28}), (\ref{29.6}), (\ref{29}) and
Sobolev Trace Theory), we have%
\begin{equation}
\left\Vert z\right\Vert _{H^{\frac{3}{2}+\epsilon }(\Omega
_{s})}+\sum\limits_{j=1}^{K}\left\Vert \nu \cdot \sigma (z)|_{\Gamma
_{j}}\right\Vert _{H^{\epsilon }(\Gamma _{j})}\leq \mathcal{O}\left( \sqrt{%
\left\vert (\Phi _{0}^{\ast },\Phi )_{\mathbf{H}}\right\vert }+\left\Vert
\Phi _{0}^{\ast }\right\Vert _{\mathbf{H}}\right).   \label{31}
\end{equation}%
Now, invoking the decomposition%
\[
w_{0}=z+\frac{i}{\beta }D([\alpha w_{0}-u_{0}-w_{0}^{\ast }]|_{\Gamma _{s}}),
\]%
we combine (\ref{28}), (\ref{31}), (\ref{29.6}) and (\ref{29}) to
have %
\begin{equation}
\left\Vert w_{0}\right\Vert _{H^{1}(\Omega
_{s})}+\sum\limits_{j=1}^{K}\left\Vert \nu \cdot \sigma (w_{0})|_{\Gamma
_{j}}\right\Vert _{H^{-\frac{1}{2}}(\Gamma _{s})}\leq \mathcal{O}\left( 
\sqrt{%
\left\vert (\Phi _{0}^{\ast },\Phi )_{\mathbf{H}}\right\vert }+\left\Vert
\Phi _{0}^{\ast }\right\Vert _{\mathbf{H}}\right).   \label{32}
\end{equation}%
\newline
\textbf{STEP III} \textbf{(An estimate for the thin elastic displacement): }For $%
1\leq j\leq K,$ we multiply both sides of the thin elastic equation in (\ref%
{27.bc})$_{2}$ by $h_{0j},$ followed by an integration over $\Gamma _{j}$,
and integration by parts. Summing the resulting expressions gives%
\[
\sum\limits_{j=1}^{K}\left[ (\sigma _{\Gamma _{s}}(h_{0j}),\epsilon _{\Gamma
_{s}}(h_{0j}))_{\Gamma _{j}}+\left\Vert h_{0j}\right\Vert ^{2}{}_{\Gamma
_{j}}\right] -\cancel {\sum\limits_{j=1}^{K}\left\langle n_{j}\cdot \sigma _{\Gamma
_{s}}(h_{0j}),h_{0j}\right\rangle _{\Gamma _{j}}}
\]%
\begin{equation}
=-(\alpha+i\beta)(u_{0},w_{0})_{\Omega _{s}}+\left\langle \nu \cdot \sigma(w_{0})  
,w_{0}\right\rangle _{\Gamma _{s}}-\left\langle \frac{\partial u_{0}}{%
\partial \nu },w_{0}\right\rangle _{\Gamma _{s}}+\sum\limits_{j=1}^{K} (h_{1j}^{*},h_{0j})_{\Gamma _{j}}.  \label{33}
\end{equation}%
(Here, in canceling the thin layer boundary terms on $\partial \Gamma _{j},$
we are invoking the boundary conditions in (\ref{27.d})-(\ref{27.e}).) Now,
with respect to the normal derivative $\frac{\partial u_{0}}{\partial \nu }%
|_{\Gamma _{s}}$ of the thermal component, we can integrate by parts so as
to deduce the trace estimate%
\begin{eqnarray*}
\left\Vert \frac{\partial u_{0}}{\partial \nu }\right\Vert _{H^{-\frac{1}{2}%
}(\Gamma _{s})} &\leq &C\left[ \left\Vert \nabla u_{0}\right\Vert _{\Omega
_{f}}+\left\Vert \Delta u_{0}\right\Vert _{\Omega _{f}}\right]  \\
&=&C\left[ \left\Vert \nabla u_{0}\right\Vert _{\Omega _{f}}+\left\Vert
(\alpha +i\beta )u_{0}-u_{0}^{\ast }\right\Vert _{\Omega _{f}}\right].
\end{eqnarray*}%
Invoking the estimate (\ref{28}) now gives%
\begin{equation}
\left\Vert \frac{\partial u_{0}}{\partial \nu }\right\Vert _{H^{-\frac{1}{2}%
}(\Gamma _{s})}=\mathcal{O}\left( \sqrt{\left\vert (\Phi _{0}^{\ast },\Phi
)_{\mathbf{H}}\right\vert }+\left\Vert \Phi _{0}^{\ast }\right\Vert _{%
\mathbf{H}}\right).   \label{34}
\end{equation}%
Applying this estimate, along with relation (\ref{28}), and (\ref{32}) for $%
\{w_{0},\nu \cdot \sigma(w_{0})|_{\Gamma _{s}}\}$ (and using
again the Sobolev Trace Theorem), we get%
\begin{equation}
\sqrt{\sum\limits_{j=1}^{K}\left[ (\sigma _{\Gamma _{s}}(h_{0j}),\epsilon
_{\Gamma _{s}}(h_{0j}))_{\Gamma _{j}}+\left\Vert h_{0j}\right\Vert
^{2}{}_{\Gamma _{j}}\right] }=\mathcal{O}\left(\sqrt{\left\Vert \Phi _{0}^{\ast }\right\Vert _{%
\mathbf{H}}\left\Vert \Phi \right\Vert _{%
\mathbf{H}}}+\sqrt{\left\vert (\Phi _{0}^{\ast
},\Phi )_{\mathbf{H}}\right\vert }+\left\Vert \Phi _{0}^{\ast }\right\Vert _{%
\mathbf{H}}\right)   \label{35}
\end{equation}%
Moreover, applying (\ref{28}) and the estimate (\ref{32}) to the resolvent relation (\ref%
{27.f}), we have %
\begin{equation}
\left\Vert w_{1}\right\Vert _{\Omega _{s}}=\mathcal{O}\left( \sqrt{\left\vert (\Phi
_{0}^{\ast },\Phi )_{\mathbf{H}}\right\vert }+\left\Vert \Phi _{0}^{\ast
}\right\Vert _{\mathbf{H}}\right) .  \label{36}
\end{equation}%
Now, if we combine (\ref{28}), (\ref{32}), (\ref{35}) and (\ref{36}), we
then have%
\begin{eqnarray*}
\left\Vert \Phi \right\Vert _{\mathbf{H}} &\leq &C_{0,\beta }\left( \sqrt{\left\Vert \Phi _{0}^{\ast }\right\Vert _{%
\mathbf{H}}\left\Vert \Phi \right\Vert _{%
\mathbf{H}}}+\sqrt{\left\vert
(\Phi _{0}^{\ast },\Phi )_{\mathbf{H}}\right\vert }+\left\Vert \Phi
_{0}^{\ast }\right\Vert _{\mathbf{H}}\right). 
\end{eqnarray*}%
Invoking finally $|ab|\leq \delta a^2+C_{\delta}b^2$ for $\delta>0$, we arrive at (after rescaling $\delta>0$),
$$(1-C_{\beta})\left\Vert \Phi \right\Vert _{\mathbf{H}}\leq C_{\beta ,\delta }\left\Vert \Phi
_{0}^{\ast }\right\Vert _{\mathbf{H}}.$$ 
This gives the required strong limit in (\ref{oniki-5}) for all $\Phi
_{0}^{\ast }\in $ $\mathbf{H}$ and all $\beta \in 
\mathbb{R}
\backslash (S\cup \{0\}).$ This completes the proof of Theorem \ref{st}.

\section{Acknowledgement}

The authors G. Avalos and Pelin G. Geredeli would like to thank the National
Science Foundation, and acknowledge their partial funding from NSF Grant
DMS-1907823.

\end{document}